\documentclass[12pt,reqno]{amsart}
\usepackage{fullpage}
\usepackage{times}
\usepackage{amsmath,amssymb,amsthm,url}
\usepackage[utf8]{inputenc}
\usepackage[english]{babel}
\usepackage{bbm}
\usepackage{enumerate}
\usepackage{bm}
\usepackage{graphicx}
\usepackage{mathrsfs}
\usepackage[colorlinks=true, pdfstartview=FitH, linkcolor=blue, citecolor=blue, urlcolor=blue]{hyperref}
\usepackage{tikz-cd}
\usepackage{tabstackengine}
\stackMath
\usepackage{comment}
\usepackage[capitalise]{cleveref}

\newtheorem{thm}{Theorem}[section]
\newtheorem{lem}[thm]{Lemma}
\newtheorem{prop}[thm]{Proposition}
\newtheorem{cor}[thm]{Corollary}
\newtheorem{conj}[thm]{Conjecture}
\newtheorem{claim}[thm]{Claim}

\theoremstyle{definition}

\newcommand{\N}{\mathbb{N}}

\newenvironment{poc}{\begin{proof}[Proof of claim]}{\end{proof}}
\title{Multiplicative irreducibility of small perturbations of the set of shifted $k$-th powers}
\author{Chi Hoi Yip}
\address{School of Mathematics\\ Georgia Institute of Technology\\ Atlanta, GA 30332\\ United States}
\email{cyip30@gatech.edu}
\subjclass[2020]{11P70, 11B30, 11D45}
\keywords{multiplicative decomposition, shifted powers}
\begin{document}

\begin{abstract}
Motivated by a conjecture of Erd\H{o}s on the additive irreducibility of small perturbations of the set of squares, recently Hajdu and S\'{a}rk\"{o}zy studied a multiplicative analogue of the conjecture for shifted $k$-th powers. They conjectured that for each $k\geq 2$, if one changes $o(X^{1/k})$ elements of $M_k'=\{x^k+1: x \in \N\}$ up to $X$, then the resulting set cannot be written as a product set $AB$ nontrivially. In this paper, we confirm a more general version of their conjecture for $k\geq 3$. 
\end{abstract} 

\maketitle

\section{Introduction}

Let $\N$ be the set of positive integers.
In this paper, we study questions related to additive decompositions and multiplicative decompositions. A set $S \subset \N$ is said to be \emph{multiplicatively reducible} if it has a multiplicative
decomposition $S=AB=\{ab: a \in A, b \in B\}$, where $A,B$ are subsets of $\N$ with size at least $2$. If $S$ is not multiplicatively reducible, we say that $S$ is \emph{multiplicatively irreducible}. Similarly, we say that $S$ is \emph{additively reducible} if it can be written as a sumset $A+B=\{a+b: a\in A, b\in B\}$, where $A,B$ are subset of $\N$ with size at least $2$. If $S$ is not additively reducible, we say that $S$ is \emph{additively irreducible}. There is a large amount of literature on the study of additive and multiplicative decompositions for sets with different arithmetic structures; we refer to a nice survey by Elsholtz \cite{E09}.

The set of perfect squares is additively irreducible simply because the gap between consecutive squares tends to infinity. Erd\H{o}s conjectured that all small perturbations of the set of squares are still additively irreducible.

\begin{conj}[Erd\H{o}s]\label{conj:Erdos}
If $k\geq 2$ and we change $o(X^{1/2})$ elements of the set of squares up to $X$ (deleting some of its elements and adding some positive integers), then the new set $R$ is always additively irreducible.    
\end{conj}

This conjecture is still open, with the best-known progress due to S\'ark\"ozy and Szemer\'edi~\cite{SS65}, who showed that~\cref{conj:Erdos} holds if one replaces $o(X^{1/2})$ with $o(X^{1/2}/ 2^{(3+\epsilon)\log X/\log \log X})$ for any $\epsilon>0$. We also refer to related results by Bienvenu~\cite{B23} and Leonetti~\cite{L23} with a probabilistic flavor. It appears that the finite field analogue of~\cref{conj:Erdos} is even more challenging (if true); we refer to \cite{HP, S12, Sh14, Y25} for some partial results. 

Recently, Hajdu and S\'{a}rk\"{o}zy studied the multiplicative decompositions of polynomial sequences with integer coefficients in a series of three papers \cite{HS18, HS18b, HS20}. In particular, in the first two papers, they provided a simple proof of the following result: for each $k\geq 2$, if one changes finitely many elements of the set $M_k'=\{x^k+1: x \in \N\}$, then the new set remains multiplicatively irreducible.  We also refer to the study of finite field analogues of the same result in \cite{KYY, S14}.

In the third paper of the series \cite{HS20}, Hajdu and S\'{a}rk\"{o}zy studied the following multiplicative analogue of~\cref{conj:Erdos}. 

\begin{conj}[Hajdu and S\'{a}rk\"{o}zy]\label{conj:main}
If $k\geq 2$ and we change $o(X^{1/k})$ elements of the set $\{x^k+1: x \in \N\}$ up to $X$, then the new set $R$ is always multiplicatively irreducible.    
\end{conj}

\cref{conj:main} is best possible in the sense that if one changes a positive proportion of elements from $\{x^k+1: x \in \N\}$ (equivalently, changes
at most $cX^{1/k}$ elements of the set $\{x^k+1: x \in \N\}$ up to $X$ for some constant $c>0$), then the resulting set could be multiplicatively reducible. For example, let $m\geq 2$ be an integer and set $R_m=AB$, where $A=\{1,m\}$ and $B=\{x^k+1: x \in \N\}$; note that $R_m$ is obtained by adding $(m^{-1/k}+o(1))X^k$ elements to $\{x^k+1: x \in \N\}$ up to $X$ and $m^{-1/k} \to 0$ as $m \to \infty$.

Based on a skillful and sophisticated argument involving various tools from Diophantine approximation, Diophantine equations, extremal graph theory, and multiplicative number theory, Hajdu and S\'{a}rk\"{o}zy \cite[Theorem 2.1]{HS20} proved a weaker version of~\cref{conj:main}. More precisely, they proved that \cref{conj:main} holds if $o(X^{1/k})$ is replaced by 
$$
o\bigg(X^{1/k} \exp\bigg(-(\log 2+\epsilon) \frac{\log X}{\log \log X}\bigg)\bigg)
$$
for any $\epsilon>0$. They remarked that both~\cref{conj:Erdos} and~\cref{conj:main} ``seem to be beyond reach in their original form". 

In this paper, we prove a more general version of~\cref{conj:main} for $k\geq 3$.

\begin{thm}\label{thm:main}
Let $k,n$ be integers with $k\geq 3$ and $n\neq 0$. If we change $o(X^{1/k})$ elements of the set $\{x^k+n: x \in \N\} \cap \N$ up to $X$, then the new set $R$ is always multiplicatively irreducible.    
\end{thm}

One key ingredient in our proof is the connection of multiplicative decompositions of the set $M_k'$ and the bipartite variants of Diophantine tuples, first considered by the author \cite{Y25b} as an attempt to make some of the results in \cite{HS18, HS18b} effective. There are many well-studied generalizations and variants of Diophantine tuples; see the recent book of Dujella \cite{D24} for an overview. The relevant variant in our setting is the following bipartite variant: for each $k \ge 2$ and each nonzero integer $n$, we call a pair of sets $(A, B)$ a \textit{bipartite Diophantine tuple with property $BD_{k}(n)$}, if $A, B$ are two subsets of $\N$ with size at least $2$, such that $ab+n$ is a $k$-th power for each $a \in A$ and $b \in B$. While this concept was only formally introduced by the author \cite{Y25b} recently, the same objects have been for example studied by Gyarmati \cite{G01}, Bugeaud and Dujella \cite{BD03}, and Bugeaud and Gyarmati \cite{BG04}, two decades ago. The connection is the following: if we can give a good absolute upper bound on $\min \{|A|, |B|\}$ among all bipartite Diophantine tuples $(A,B)$ with property $BD_{k}(-n)$, then it seems plausible that the following K\"ov\'ari-S\'os-Tur\'an theorem \cite{KST54} from extremal graph theory can be used to make some partial progress on~\cref{conj:main}.  

\begin{lem}[K\"ov\'ari-S\'os-Tur\'an theorem]\label{lem:KST}\
Let $G$ be a bipartite graph with vertex classes $U$ and $V$  such that $|U|=m$ and $|V|=n$. Assume that there do not exist a set $X\subset U$ with size $s$ and a set $Y\subset V$ with size $t$, such that $x$ and $y$ are adjacent for all $x\in X$ and $y\in Y$. Then the number of edges of $G$ is at most 
$(s-1)^{1/t}(n-t+1)m^{1-1/t}+(t-1)m.$
\end{lem}

Our proof techniques cannot handle the case $k=2$ in~\cref{conj:main}. In fact, it is an open question to show unconditionally that if $(A, B)$ is a bipartite Diophantine tuple with property $BD_{2}(1)$ (or $BD_2(-1)$, resp.), then $\min \{|A|,|B|\}$ is bounded by an absolute constant \cite{BHP25, BG04}; note that this would follow easily if we assume the uniformity conjecture \cite{CHM}. As for the case $k\geq 3$, the same question was partially addressed by the author \cite{Y25b}; in particular, it was shown in \cite[Theorem 2.2]{Y25b} that if $k\geq 3$ is fixed, and $|n|\to \infty$, then for any bipartite Diophantine tuple $(A, B)$ with property $BD_{k}(n)$, we have $\min \{|A|,|B|\}\ll_{k} \log |n|$. Clearly, such a result is not strong enough for the application to \cref{conj:main}. Instead, we realize that studying a ``local version" of bipartite Diophantine tuples is sufficient to prove~\cref{conj:main} for $k\geq 3$.

\section{Forbidden local structures}\label{sec:large}

In this section, we prove some refined ``local estimates" on bipartite Diophantine tuples with some carefully chosen parameters. We list several results from \cite{Y25b} and deduce some useful corollaries. We first introduce the following constants defined in \cite{Y25b}:
\begin{align*}
s_3=6, \quad  s_4=4, \quad s_5=3, \quad \text{ and } \quad s_k=2 \quad \text{ for } k \geq 6;
\end{align*}
$$t_3=\frac{15399}{938}, \quad t_4=\frac{34}{3}, \quad t_5=\frac{97}{23}, \quad t_6=\frac{29}{4}, \text{ and } \quad t_k=\frac{k^2+k-4}{k^2-6k+6} \quad \text{ for } k \geq 7.
$$

The following proposition is one of the key results in \cite{Y25b}. Its proof is based on a combination of an explicit version of the bounds of the number of solutions of Thue inequalities, repeated applications of gap principles, and some combinatorial arguments. 

\begin{prop}[{\cite[Proposition 4.1]{Y25b}}]\label{prop:a1a2}
Let $k \geq 3$ and let $n$ be a nonzero integer. Let $A,B \subset \N$ such that $A=\{a_1, a_2, \ldots, a_{\ell}\}$ and $B=\{b_1, b_2, \ldots, b_m\}$ with $a_1<a_2<\cdots<a_{\ell}$ and $b_1<b_2<\cdots<b_m$, and $AB+n \subset \{x^k: x \in \N\}$. If $k>3$, further assume that $m \geq s_k+1$, $\ell \geq 2$, and $a_2 \leq b_{m-s_k}$; if $k=3$, further assume that $m \geq 7$, $\ell \geq 3$, and $a_3 \leq b_{m-6}$. Then at most $s_k$ elements in $B$ are at least $2|n|^{t_k}$.    
\end{prop}

For our purpose, we shall use the following immediate corollary of  Proposition~\ref{prop:a1a2}. Note that $t_k\leq 17$ and $s_k\leq 6$ for all $k\geq 3$.

\begin{cor}\label{cor:K23weak}
Let $k \geq 3$ and $n$ be a nonzero integer. There do not exist positive integers $a_1,a_2,a_3$ and $b_1, b_2, \ldots, b_{7}$ with $a_1<a_2<a_3\leq b_1<b_2<\ldots<b_{7}$ and $2|n|^{17} \leq b_1$, such that $a_ib_j+n$ is a $k$-th power for all $1\leq i \leq 3$ and $1\leq j \leq 7$.    
\end{cor}

The following lemma is based on a simple gap principle.

\begin{lem}[{\cite[Lemma 3.6]{Y25b}}]\label{gap_principle}
Let $k \geq 3$ and let $n$ be a nonzero integer. Let $a,b,c,d$ be positive integers such that $a<b$, $c<d$, and $ac\geq 2|n|$. Suppose further that $ac+n, bc+n, ad+n, bd+n$ are $k$-th powers. Then $bd \geq k^{k} (ac)^{k-1}/(4^{k-1}|n|^k)$.     
\end{lem}

Next, we deduce two corollaries of \cref{gap_principle}.

\begin{cor}\label{cor:gap}
Let $k \geq 3$ and let $n$ be a nonzero integer. If $X>4^{6(k-1)}|n|^{6k}$, then there do not exist integers $a_1,a_2,b_1,b_2$ with $X^{1/3}<a_1<a_2\leq X^{1/2}<b_1<b_2\leq X$ such that $a_ib_j+n$ are $k$-th powers for all $1\leq i,j \leq 2$.   
\end{cor}
\begin{proof}
Suppose otherwise that there do exist $a_1,a_2,b_1,b_2$ with the required property. Note that we have $a_1b_1\geq X^{5/6}>2|n|$. Then by Lemma~\ref{gap_principle}, we have $a_2b_2\geq k^{k}(a_1b_1)^{k-1}/(4^{k-1}|n|^k)$. It follows from the assumptions $X>4^{6(k-1)}|n|^{6k}$ and $X^{1/3}<a_1<a_2\leq X^{1/2}<b_1<b_2\leq X$ that
$$
b_2\geq \frac{k^{k}(a_1b_1)^{k-1}}{4^{k-1}|n|^k a_2}\geq \frac{1}{4^{k-1}|n|^k } \cdot \frac{a_1^2b_1^2}{a_2}\geq \frac{1}{4^{k-1}|n|^k } \cdot X^{1/6} b_1^2\geq b_1^2>X,
$$
violating the assumption that $b_2\leq X$.
\end{proof}

\begin{cor}\label{cor:gap2}
Let $k \geq 3$ and let $n$ be a nonzero integer. Let $a_1,a_2$ be positive integers with $a_1<a_2$. Then for $X$ sufficiently large, the number of positive integers $b\leq X$ such that $a_1b+n$ and $a_2b+n$ are both $k$-th powers are at most $2\log \log X$.
\end{cor}
\begin{proof}
Suppose that $b_1<b_2<\cdots<b_m<X$ are positive integers such that $a_1b_i+n$ and $a_2b_i+n$ are both $k$-th powers for each $1\leq i \leq m$. For each $1\leq i \leq m-1$, applying~\cref{gap_principle} to $a_1, a_2, b_i, b_{i+1}$, we get
$$
b_{i+1}\geq \frac{k^{k}(a_1b_i)^{k-1}}{4^{k-1}|n|^k a_2}\geq Cb_i^2,
$$
where $C$ is a positive constant depending on $k,n,a_1,a_2$. It follows that for each $1\leq i \leq m-1$, we have $\log b_{i+1}\geq 2\log b_i+\log C$ and thus $\log b_{i+1}+\log C\geq 2(\log b_i+\log C)$. Choose the smallest $j$ such that $\log b_j+\log C\geq 1$; note that if such $j$ does not exist, then we have $m\leq e/C$ and we are already done. Since $b_j\geq j$, we have $j\leq e/C$. It follows that
$$
\log X+\log C\geq \log b_m+\log C\geq 2^{m-j}(\log b_j+\log C)\geq 2^{m-j} \geq 2^{m-e/C}
$$
and thus $m\leq 2\log \log X$ for $X$ sufficiently large.
\end{proof}

\section{Proof of the main result}

Our proof is inspired by several arguments used in \cite{HS20, Y25b}. We shall use the following standard notation: for each set $A\subset \N$ and each positive real number $X$, we write $$A(X)=|\{x \in A: x \leq X\}|.$$

\begin{proof}[Proof of Theorem~\ref{thm:main}]
Write $M=\{x^k+n: x \in \N\} \cap \N$. We can write $R=Q \cup S$, where $Q \subset M$ and $S\cap M=\emptyset$. By the assumption on $R$, we have $Q(X)=(1-o(1))X^{1/k}$ and $S(X)=o(X^{1/k})$. For the sake of contradiction, suppose that $R$ is multiplicatively reducible. Then we can write $R=A \cdot B$, where $A,B$ are subsets of $\N$ with size at least $2$. 

\begin{claim}\label{claim2}
Let $X_0\geq \max\{4|n|^{34},4^{6(k-1)}|n|^{6k}\}$ be a sufficiently large number such that $Q(X)\geq \frac{1}{2}X^{1/k}$ whenever $X\geq X_0$. If $X\geq X_0$, then one of the following holds:
\begin{enumerate}
    \item $\max \{A(X), B(X)\}\geq \frac{1}{40}X^{1/k}$;
    \item $\max \{A(X^{1/2}), B(X^{1/2})\}\geq \frac{1}{40}X^{1/2k}$;
    \item $\max \{A(X^{1/3}), B(X^{1/3})\}\geq \frac{1}{40}X^{1/3k}$. 
\end{enumerate}    
\end{claim}
\begin{poc}
Let $X\geq X_0$ be fixed. Write $$A_1=A \cap [1,X^{1/3}], \quad  A_2=A \cap (X^{1/3}, X^{1/2}], \quad A_3=A \cap (X^{1/2},X].$$
Similarly, write
$$
B_1=B \cap [1,X^{1/3}], \quad  B_2=B \cap (X^{1/3}, X^{1/2}], \quad B_3=B \cap (X^{1/2},X].
$$

Each integer $x \in R \cap [1,X]$ can be written as $x=ab$ for some $a\in A$ and $b\in B$ with $1\leq a,b\leq x$; note that we always have either $a\leq \sqrt{x}$ or $b\leq \sqrt{x}$. Thus, 
\begin{align*}
Q \cap [1,X]
&\subset R \cap [1,X] \\
&\subset ((A_1 \cup A_2) \cdot (B_1 \cup B_2 \cup B_3)) \cup ((A_1 \cup A_2 \cup A_3) \cdot (B_1 \cup B_2))\\
&=((A_1 \cup A_2) \cdot (B_1 \cup B_2)) \cup (A_1 \cdot B_3) \cup (A_2 \cdot B_3) \cup (B_1 \cdot A_3) \cup (B_2 \cdot A_3).
\end{align*}
Observe that if $ab\in Q$, then $ab-n$ is a $k$-th power. Next, we consider the contribution of the five sets on the right-hand side of the above equation to $Q \cap [1,X]$. 
\begin{enumerate}
    \item Clearly, the number of pairs $(a,b) \in (A_1 \cup A_2) \times (B_1 \cup B_2)$ such that $ab \in Q$ is at most $|A_1 \cup A_2||B_1 \cup B_2|=A(X^{1/2})B(X^{1/2})$. 
    \item Build a bipartite graph $G$ with vertex classes $A_1$ and $B_3$ such that for each $a\in A_1$, and $b\in B_3$, there is an edge between $a$ and $b$ if and only if $ab-n$ is a $k$-th power. Then the number of pairs $(a,b) \in A_1 \times B_3$ such that $ab\in Q$ is at most the number of edges of $G$. 
Since $X_0\geq 4|n|^{34}$, we have $\min_{b \in B_3} b\geq \sqrt{X}\geq \sqrt{X_0}=2|n|^{17}$. Also note that we have $\max_{a\in A_1} a\leq X^{1/3}<X^{1/2}\leq \min_{b\in B_3}b$. Thus, by Corollary~\ref{cor:K23weak}, there do not exist three distinct elements $a_1,a_2,a_3\in A_1$, and seven distinct elements $b_1,b_2,\cdots,b_7\in B_3$ such that $a_ib_j-n$ is a $k$-th power for each $i\in \{1,2\}$ and $1\leq j \leq 7$. Thus, applying the K\"ov\'ari-S\'os-Tur\'an theorem (Lemma~\ref{lem:KST}) with $U=B_3, V=A_1, s=7, t=3$,  the number of edges of $G$ is at most 
$$\sqrt[3]{6}|A_1||B_3|^{1-1/3}+2|B_3|\leq 2A(X^{1/3})B(X)^{2/3}+2B(X).$$
\item Since $X_0\geq 4^{6(k-1)}|n|^{6k}$, Corollary~\ref{cor:gap} implies that there do not exist two distinct elements $a_1,a_2\in A_2$, and two distinct elements $b_1,b_2\in B_3$ such that $a_ib_j-n$ is a $k$-th power for each $1\leq i,j \leq 2$. Thus, similar to the analysis in (2), the K\"ov\'ari-S\'os-Tur\'an theorem implies that the number of pairs $(a,b) \in A_2 \times B_3$ such that $ab\in Q$ is at most 
$\sqrt{2}|A_2||B_3|^{1-1/2}+|B_3|\leq 2A(X^{1/2})B(X)^{1/2}+B(X).$
    \item Similarly to (2), the number of pairs $(b, a) \in B_1 \times A_3$ such that $ab\in Q$ is at most $2B(X^{1/3})A(X)^{2/3}+2A(X)$. 
    \item Similarly to (3), the number of pairs $(b, a) \in B_2 \times A_3$ such that $ab\in Q$ is at most $2B(X^{1/2})A(X)^{1/2}+A(X)$. 
\end{enumerate}
It follows that 
\begin{align}
Q(X) 
&\leq
A(X^{1/2})B(X^{1/2})+2A(X^{1/3})B(X)^{2/3}+2B(X)+2A(X^{1/2})B(X)^{1/2}+B(X) \notag\\
&\quad +2B(X^{1/3})A(X)^{2/3}+2A(X)+2B(X^{1/2})A(X)^{1/2}+A(X) \label{eq:Q(X)}.
\end{align}
Suppose the statement of the claim is false; then inequality~\eqref{eq:Q(X)} implies that
\begin{align*}   
Q(X)
&\leq \bigg(\frac{1}{40}+\frac{2}{40}+\frac{2}{40}+\frac{2}{40}+\frac{1}{40}+\frac{2}{40}+\frac{2}{40}+\frac{2}{40}+\frac{1}{40}\bigg)X^{1/k}<\frac{1}{2}X^{1/k},
\end{align*}
contradicting the assumption that $Q(X)\geq \frac{1}{2}X^{1/k}$. This finishes the proof of the claim.
\end{poc}

By~\cref{claim2}, there is an increasing sequence $(X_j)_{j=1}^{\infty}$ with $X_j \to \infty$, such that $$\max \{A(X_j), B(X_j)\}\geq \frac{1}{40}X_j^{1/k}$$ for each $j\in \N$. Without loss of generality, by passing to a subsequence, we may assume that $B(X_j)\geq \frac{1}{40}X_j^{1/k}$ for each $j\in \N$. Since $|A|\geq 2$, so we can pick two fixed elements $a_1,a_2\in A$ with $a_1<a_2$. For each $b\in B \cap [1,X]$, we have $a_1b<a_2b\leq a_2X$. Thus, the number of $b\in B \cap [1,X]$ with $a_ib \in S$ for some $i \in \{1,2\}$ is at most $S(a_2X)=o((a_2X)^{1/k})=o(X^{1/k})$. It follows that the number of $b\in B \cap [1,X_j]$ with $a_ib \in Q$ (in particular, $a_ib-n$ is a $k$-th power) for $i \in \{1,2\}$ is at least $$B(X_j)-o(X_j^{1/k})\geq \frac{1}{40}X_j^{1/k}-o(X_j^{1/k})=(\frac{1}{40}-o(1))X_j^{1/k}>\frac{1}{80}X_j^{1/k}>2\log \log X_j$$ for all $j$ large enough. However, this contradicts~\cref{cor:gap2}. 
\end{proof}

\section*{Acknowledgments}
The author thanks Ernie Croot, Bruce Landman, Greg Martin, and Jiaxi Nie for helpful discussions. The author also thanks the anonymous referees for their valuable comments and suggestions.

\bibliographystyle{abbrv}
\bibliography{main}

\begin{thebibliography}{10}

\bibitem{BHP25}
G.~o. Batta, L.~Hajdu, and A.~Pongr\'acz.
\newblock On {D}iophantine graphs.
\newblock {\em J. Lond. Math. Soc. (2)}, 111(5):Paper No. e70163, 2025.

\bibitem{B23}
P.-Y. Bienvenu.
\newblock Metric decomposability theorems on sets of integers.
\newblock {\em Bull. Lond. Math. Soc.}, 55(6):2653--2659, 2023.

\bibitem{BD03}
Y.~Bugeaud and A.~Dujella.
\newblock On a problem of {D}iophantus for higher powers.
\newblock {\em Math. Proc. Cambridge Philos. Soc.}, 135(1):1--10, 2003.

\bibitem{BG04}
Y.~Bugeaud and K.~Gyarmati.
\newblock On generalizations of a problem of {D}iophantus.
\newblock {\em Illinois J. Math.}, 48(4):1105--1115, 2004.

\bibitem{CHM}
L.~Caporaso, J.~Harris, and B.~Mazur.
\newblock Uniformity of rational points.
\newblock {\em J. Amer. Math. Soc.}, 10(1):1--35, 1997.

\bibitem{D24}
A.~Dujella.
\newblock {\em {D}iophantine $m$-tuples and {E}lliptic Curves}, volume~79 of {\em Developments in Mathematics}.
\newblock Springer, Cham, 2024.

\bibitem{E09}
C.~Elsholtz.
\newblock A survey on additive and multiplicative decompositions of sumsets and of shifted sets.
\newblock In {\em Combinatorial number theory and additive group theory}, Adv. Courses Math. CRM Barcelona, pages 213--231. Birkh\"{a}user Verlag, Basel, 2009.

\bibitem{G01}
K.~Gyarmati.
\newblock On a problem of {D}iophantus.
\newblock {\em Acta Arith.}, 97(1):53--65, 2001.

\bibitem{HS18}
L.~Hajdu and A.~S\'{a}rk\"{o}zy.
\newblock On multiplicative decompositions of polynomial sequences, {I}.
\newblock {\em Acta Arith.}, 184(2):139--150, 2018.

\bibitem{HS18b}
L.~Hajdu and A.~S\'{a}rk\"{o}zy.
\newblock On multiplicative decompositions of polynomial sequences, {II}.
\newblock {\em Acta Arith.}, 186(2):191--200, 2018.

\bibitem{HS20}
L.~Hajdu and A.~S\'{a}rk\"{o}zy.
\newblock On multiplicative decompositions of polynomial sequences, {III}.
\newblock {\em Acta Arith.}, 193(2):193--216, 2020.

\bibitem{HP}
B.~Hanson and G.~Petridis.
\newblock Refined estimates concerning sumsets contained in the roots of unity.
\newblock {\em Proc. Lond. Math. Soc. (3)}, 122(3):353--358, 2021.

\bibitem{KYY}
S.~{Kim}, C.~H. {Yip}, and S.~{Yoo}.
\newblock {Multiplicative structure of shifted multiplicative subgroups and its applications to Diophantine tuples}.
\newblock Canad. J. Math., to appear. arXiv:2309.09124.

\bibitem{KST54}
T.~K\"ovari, V.~T. S\'os, and P.~Tur\'an.
\newblock On a problem of {K}. {Z}arankiewicz.
\newblock {\em Colloq. Math.}, 3:50--57, 1954.

\bibitem{L23}
P.~Leonetti.
\newblock Almost all sets of nonnegative integers and their small perturbations are not sumsets.
\newblock {\em Proc. Amer. Math. Soc.}, 151(9):3681--3689, 2023.

\bibitem{S12}
A.~S\'{a}rk\"{o}zy.
\newblock On additive decompositions of the set of quadratic residues modulo {$p$}.
\newblock {\em Acta Arith.}, 155(1):41--51, 2012.

\bibitem{S14}
A.~S\'{a}rk\"{o}zy.
\newblock On multiplicative decompositions of the set of the shifted quadratic residues modulo {$p$}.
\newblock In {\em Number theory, analysis, and combinatorics}, De Gruyter Proc. Math., pages 295--307. De Gruyter, Berlin, 2014.

\bibitem{SS65}
A.~S\'ark\"ozy and E.~Szemer\'edi.
\newblock On the sequence of squares.
\newblock {\em Mat. Lapok}, 16:76--85, 1965.

\bibitem{Sh14}
I.~D. Shkredov.
\newblock Sumsets in quadratic residues.
\newblock {\em Acta Arith.}, 164(3):221--243, 2014.

\bibitem{Y25}
C.~H. Yip.
\newblock Restricted sumsets in multiplicative subgroups.
\newblock Canad. J. Math., published online 2025:1--25. \url{https://doi.org/10.4153/S0008414X24000920}.

\bibitem{Y25b}
C.~H. Yip.
\newblock Multiplicatively reducible subsets of shifted perfect {$k$}th powers and bipartite {D}iophantine tuples.
\newblock {\em Acta Arith.}, 218(3):251--271, 2025.

\end{thebibliography}

\end{document}